\documentclass[11pt]{amsart}
\usepackage{amssymb}
\usepackage{stmaryrd}

\newtheorem{thm}{Theorem}[section]
\newtheorem{lemma}[thm]{Lemma}
\newtheorem{proposition}[thm]{Proposition}

\newtheorem{definition}[thm]{Definition}
\newtheorem{corollary}[thm]{Corollary}

\newtheorem{question}[thm]{Question}

\newcommand{\dom}{\mathrm{dom}}

\newcommand{\h}{\mathrm{ht}}

\newcommand{\res}{\upharpoonright}

\newcommand{\ka}{\kappa}

\newcommand{\pro}{\prod_{\omega,\kappa} R}

\begin{document}

\title{A Rigid Kurepa Tree From a Free Suslin Tree}

\author{John Krueger}

\address{John Krueger \\ Department of Mathematics \\ 
	University of North Texas \\
	1155 Union Circle \#311430 \\
	Denton, TX 76203}
\email{jkrueger@unt.edu}

\date{March 2023; Revised January 2025.}



\begin{abstract}
We analyze a countable support product of a free Suslin tree which turns it into a 
highly rigid Kurepa tree with no Aronszajn subtree. 
In the process, we introduce a new rigidity property for trees, which says roughly speaking 
that any non-trivial strictly increasing function from a section of the tree into itself maps 
into a cofinal branch.
\end{abstract}

\maketitle

\section{Introduction}

In this article we introduce a method of turning a free Suslin tree 
into a Kurepa tree by a countable support product forcing. 
The standard way to add a cofinal branch to a Suslin tree is by forcing with the tree itself 
(with the order reversed). 
There are a variety of possible outcomes after forcing with a Suslin tree, ranging from 
adding exactly one cofinal branch, as is the case with a free Suslin tree, 
to adding more than $\omega_1$ many cofinal branches. 
In particular, a Suslin tree which becomes a Kurepa tree after forcing with it 
is sometimes called an almost Kurepa Suslin tree 
in the literature, and some recent examples can be found in 
\cite{fuchs} and \cite{ramandi}.

A more direct albeit naive way to try to turn a Suslin tree into a Kurepa tree 
would be to force with the tree repeatedly enough times to produce the desired 
number of branches. 
However, in some cases even adding 
two cofinal branches to a Suslin tree will collapse $\omega_1$. 
This happens, for example, when the tree is self-specializing, which means that adding a 
cofinal branch specializes the rest of the tree outside of the branch. 
On the other hand, if a Suslin tree $S$ is $2$-free, then after forcing with it 
there exists a dense set of $x \in S$ for which $S_x$ (the part of the tree above $x$) 
is still Suslin. 
More broadly, for any $n < \omega$, a normal Suslin tree $S$ is $(n+1)$-free iff after 
forcing with $S$ $n$-many times, there are densely many $x$ in $S$ such that 
$S_x$ is Suslin.

These facts suggest considering free Suslin trees as candidates for a type of Suslin tree 
which could turn into a Kurepa tree after repeatedly forcing with it via 
a product forcing. 
As a first attempt, one could consider a finite support product of a free Suslin tree, 
with the idea that the product might be c.c.c. 
However, by a result of Jensen and Schlechta \cite{jensen}, 
after L{\'e}vy collaping 
a Mahlo cardinal to become $\omega_2$, there does not exist any c.c.c.\ forcing 
which adds a Kurepa tree, and free Suslin trees can exist in such a model.

It turns out that a countable support product of a free Suslin tree 
is a reasonable and effective forcing for turning a Suslin tree into a Kurepa tree. 
Such a product forcing is proper, 
countably distributive, and $(2^\omega)^+$-c.c. 
In particular, assuming \textsf{CH}, it preserves all cardinals (and in fact, 
in the absence of \textsf{CH} it collapses $2^\omega$ to become $\omega_1$). 
In addition, the Kurepa tree thus obtained has nice properties. 
For example, it has no Aronszajn subtree, and the cofinal branches of the tree 
are exactly the generics for the factors of the product forcing.

But the standout feature of this Kurepa tree is a very strong rigidity property 
which in particular 
implies the non-existence of injective strictly increasing and level preserving 
maps from the tree into itself other than 
the identity function. 
More specifically, after forcing with a countable support product of a free Suslin tree $R$, 
any strictly increasing map from a dense subset of $R_x$ into $R_y$, 
where $x$ and $y$ are incomparable, collapses into a branch above densely many 
elements of its domain.

\section{Preliminaries}

All of the trees in this article are assumed to have height $\omega_1$. 
The reader should be familiar with Aronszajn and Suslin trees, 
as well as the basics of forcing. 
We describe some of the notation and terminology which we will use, 
and prove two known lemmas which we will need later.

An \emph{$\omega_1$-tree} is a tree of height $\omega_1$ whose levels are countable. 
A \emph{Kurepa tree} is an $\omega_1$-tree with at least $\omega_2$ 
many cofinal branches. 
If $b$ is a cofinal branch of a tree $T$, we will write $b(\alpha)$ for the 
unique element of $b$ of height $\alpha$. 
Let $\h_T(x)$ denote the height of $x$ in the tree $T$, 
and for any $\gamma \le \h_T(x)$, 
$x \res \gamma$ is the unique element of $T$ below $x$ of height $\gamma$. 
Let $T_\alpha$ denote level $\alpha$ of $T$, which is 
the set of elements of $T$ with height $\alpha$, and let 
$T \res \beta := \bigcup_{\alpha < \beta} T_\alpha$.

A tree is \emph{normal} if it has a root, every element has at least two 
immediate successors, any distinct elements of the same limit height have 
different sets of elements below them, and every element has elements above 
it at any higher level. 
A \emph{subtree} of a tree $T$ is any subset of $T$ considered as a tree 
with the order induced by $<_T$ 
(we note that this differs from the terminology of 
some authors who require a subtree to be downwards closed). 
The \emph{downward closure} of a set $U \subseteq T$ is the set of $x$ such that for some 
$y \in U$, $x \le_T y$. 
For any $x \in T$, $T_x$ is the subtree $\{ y \in T : x \le_T y \}$.

For trees $T_0, \ldots, T_{n-1}$, 
$T_0 \otimes \cdots \otimes T_{n-1}$ is the tree consisting of all 
$n$-tuples $(a_0,\ldots,a_{n-1}) \in T_0 \times \cdots \times T_{n-1}$ 
such that $a_0,\ldots,a_{n-1}$ have the same heights in their respective trees, 
ordered componentwise. 
An \emph{$n$-derived tree} of a tree $T$ is a tree of the form 
$T_{x_0} \otimes \cdots \otimes T_{x_{n-1}}$, where $x_0, \ldots, x_{n-1}$ 
are distinct elements of $T$ of the same height. 
A tree $T$ is \emph{$n$-free} if all of its $n$-derived trees are Suslin, and 
is \emph{free} if it is $n$-free for all $0 < n < \omega$.

When we consider a Suslin tree as a forcing notion, we implicitly mean the forcing 
with the tree order reversed, and use standard forcing terminology such as ``dense'' or 
``dense open'' with the reverse order. 
If $T$ is a Suslin tree and $D \subseteq T$ is dense open, then there exists some 
$\gamma < \omega_1$ such that $T_\gamma \subseteq D$. 
When we refer to a cardinal $\theta$ being ``large enough'', we mean that it 
is large enough so that $H(\theta)$ contains all parameters under discussion.

\begin{lemma}
	Assume that $S$ is a Suslin tree. 
	Let $\theta$ be a large enough regular cardinal and let 
	$M$ be a countable elementary substructure of $H(\theta)$ with $S \in M$. 
	Suppose that $X \in M$ is a subset of $S$, $x \in S_{M \cap \omega_1}$, 
	and there exists some $z \in X$ with $x \le_S z$. 
	Then there exists some $y <_S x$ with $y \in X$.
\end{lemma}

\begin{proof}
	Define $D$ as the set of $a \in S$ such that either there exists some $y \in X$ 
	with $y \le_S a$, 
	or else $a$ is incomparable with every member of $X$. 
	It is easy to check that $D$ is dense open in $S$. 
	So there exists a least $\gamma < \omega_1$ such that $S_{\gamma} \subseteq D$.  
	Note that by elementarity $D \in M$, so $\gamma < M \cap \omega_1$, 
	and $x \res \gamma$ is in $D$. 
	Since $x \res \gamma <_S x \le_S z \in X$, by the definition of $D$ 
	there exists some $y \in X$ such that $y \le_S x \res \gamma$, and 
	so $y <_S x$.
\end{proof}

\begin{lemma}
	Let $T$ be an $\omega_1$-tree. 
	If $T$ has an Aronszajn subtree, then $T$ has a downwards closed 
	Aronszajn subtree.
\end{lemma}

\begin{proof}
	Assume that $U$ is an uncountable subtree of $T$ with countable levels, 
	and let $W$ be the downward closure of $U$. 
	We claim that if $U$ is Aronszajn, then so is $W$, or equivalently, 
	if $W$ has a cofinal branch, then so does $U$. 
	Suppose that $b$ is a cofinal branch of $W$. 
	We will prove that $b \cap U$ is uncountable, which implies that 
	$U$ contains a cofinal branch.

	Suppose for a contradiction that $b \cap U$ is countable. 
	Then there exists some $\alpha < \omega_1$ 
	such that $b \cap U \subseteq T \res \alpha$. 
	As $U$ has countable levels, 
	we can find some $\beta < \omega_1$ greater than $\alpha$ such that 
	$U \res (\alpha+1) \subseteq T \res \beta$.

	Since $b$ is uncountable, fix some $x \in b$ with $\h_T(x) > \beta$. 
	Then $x \in W$, so let $y$ be of least height such that 
	$y \in U$ and $x \le_T y$. 
	Then the height of $y$ in $T$ is greater than $\beta$, 
	which implies that $y$ is not in $U \res (\alpha+1)$.

	Let $z^*$ be the $\alpha$-th member (according to the tree ordering) 
	of the set $\{ z \in U : z <_U y \}$. 
	Then obviously the height of $z^*$ in $T$ is at least $\alpha$. 
	As $b \cap U \subseteq T \res \alpha$ and $z^* \in U$, 
	it follows that $z^* \notin b$. 
	Now $x \le_T y$ and $z^* <_T y$, so $x$ and $z^*$ are comparable in $T$. 
	By the minimal height of $y$, $z^* <_T x$. 
	But $x \in b$, so $z^* \in b$ and we have a contradiction. 
\end{proof}

\section{Turning a Free Tree Into a Kurepa Tree}

For the remainder of the article we assume that 
$R$ is a normal free Suslin tree and $\ka$ is a non-zero ordinal number. 
The existence of such a tree follows from $\Diamond$ (see \cite[Chapter V]{dj}). 
We will define a forcing poset which adds no new countable sets of ordinals, 
is $(2^\omega)^+$-c.c., 
forces \textsf{CH}, and adds distinct cofinal branches $b_i$ to $R$ 
for each $i < \ka$ (and those are all of the cofinal branches). 
In particular, when $\ka \ge \omega_2$, then 
$R$ is forced to become a Kurepa tree. 
We will also show that in the forcing extension, $R$ has no Aronszajn subtree 
and satisfies a strong rigidity property. 
In fact, the forcing is just a countable support product of $\ka$-many 
copies of $R$.

\begin{definition}
	Let $\pro$ be the forcing poset consisting of all functions 
	$p : \dom(p) \to R$ 
	such that $\dom(p)$ is a countable subset of $\ka$ and for each 
	$i \in \dom(p)$, $p(i) \in R$. 
	Define $q \le p$ if $\dom(p) \subseteq \dom(q)$ and for all 
	$i \in \dom(p)$, $p(i) \le_R q(i)$.
\end{definition}

The following lemmas are straightforward to check.

\begin{lemma}
	For any $i < \kappa$ and $\gamma < \omega_1$, 
	the set of $p \in \pro$ 
	such that $i \in \dom(p)$ and the height of $p(i)$ in $R$ 
	is greater than $\gamma$ is dense open in $\pro$.
\end{lemma}

\begin{lemma}
	For all distinct $i$ and $j$ in $\ka$, 
	the set of $p \in \pro$ such that $i$ and $j$ 
	are in $\dom(p)$ and $p(i)$ and $p(j)$ are incomparable in $R$ is dense open.
\end{lemma}

\begin{lemma}
	If $\kappa \ge \omega_1$, then for all $x \in R$, 
	the set of $p \in \pro$ such that for some $i \in \dom(p)$, 
	$x <_R p(i)$, is dense open.
\end{lemma}

Suppose that $G$ is a generic filter on $\pro$. 
In $V[G]$, define for each $i < \kappa$ 
$$
b_i := \{ y \in R : \exists p \in G, \ i \in \dom(p) \ \land \ y \le_R p(i) \}.
$$
A straightforward argument using the previous lemmas shows that in $V[G]$, 
each $b_i$ is a cofinal branch of $R$, 
and for any distinct $i$ and $j$ in $\ka$, $b_i \ne b_j$. 
Consequently, in $V[G]$ the tree $R$ has at least $|\kappa|$ many cofinal branches. 
Moreover, if $\ka \ge \omega_1$, then $\bigcup_{i < \ka} b_i = R$. 
Let $\dot b_i$ be a name for the branch $b_i$ described above, 
and conversely, whenever $G$ is a generic filter on $\pro$, 
$b_i$ will denote $\dot b_i^G$.

\begin{proposition}
	The forcing poset $\pro$ is $(2^\omega)^+$-Knaster.
\end{proposition}

\begin{proof}
	Let $\{ p_i : i < (2^\omega)^+ \}$ be a family of conditions. 
	Applying the $\Delta$-system lemma to the domains of the conditions, 
	we can find an unbounded set $X \subseteq (2^\omega)^+$ and 
	a countable set $r$ such that for all $i < j$ in $X$, 
	$\dom(p_i) \cap \dom(p_j) = r$ 
	(see \cite[Lemma III.6.15]{kunen}, namely, the case when $\lambda = \omega_1$ and 
	$\ka = (2^\omega)^+$). 
	Now each $p_i \res r$ is a function from a countable set into a set of size $\omega_1$, 
	so there are at most $\omega_1^\omega = 2^\omega$ many possibilities for 
	such a function. 
	So find a set $Y \subseteq X$ of size $(2^\omega)^+$ such that 
	for all $i < j$ in $Y$, $p_i \res r = p_j \res r$. 
	It is easy to see that for all $i < j$ in $Y$, $p_i \cup p_j$ is a condition 
	below both $p_i$ and $p_j$.
\end{proof}

Now we start working toward showing that $\pro$ is proper and countably distributive. 
The following lemma is easy.

\begin{lemma}
	Let $\theta$ be a large enough regular cardinal, $M \prec H(\theta)$ countable, 
	and assume that $\pro \in M$.  
	Let $\delta := M \cap \omega_1$. 
	Suppose that $\langle p_n : n < \omega \rangle$ is a descending sequence of 
	conditions in $M \cap \pro$ and $\langle z_n : n < \omega \rangle$ is 
	a sequence of elements of $R_\delta$ satisfying:
	$$
	\forall i \in \bigcup_n \dom(p_n) \ 
	\exists k_i < \omega \ \exists n_i < \omega \ 
	\forall m \ge n_i \ p_m(i) <_R z_{k_i}.
	$$
	Define $q$ to be the function with domain equal to 
	$\bigcup_n \dom(p_n)$ 
	satisfying that for all $i \in \dom(q)$, 
	$q(i) := z_{k_i}$. 
	Then $q \in \pro$ and $q \le p_n$ for all $n$.
\end{lemma}

\begin{lemma}
	Let $\theta$ be a large enough regular cardinal, $M \prec H(\theta)$ countable, 
	and assume that $\pro \in M$. 
	Let $\delta := M \cap \omega_1$. 
	Suppose that $D \in M$ is a dense open subset of $\pro$, 
	$p \in M \cap \pro$, $i_0,\ldots,i_{n-1} \in \dom(p)$, and 
	$z_0,\ldots,z_{n-1}$ are distinct elements of $R_\delta$ such that 
	$p(i_k) <_R z_k$ for all $k < n$.
	Then there exists some $r \le p$ in $M \cap D$ such that 
	for all $k < n$, $r(i_k) <_R z_k$.
\end{lemma}

\begin{proof}
	Since $R$ is normal, we can find some 
	$\gamma < \delta$ so that for all $k < m < n$, 
	$z_{k} \res \gamma \ne z_{m} \res \gamma$, and 
	$\gamma$ is greater than the height of $p(i_k)$ for all $k < n$. 
	Define $q$ with the same domain as $p$ by 
	letting $q(i_k) := z_k \res \gamma$ for all $k < n$, 
	and $q(j) := p(j)$ for any other $j$ in $\dom(p)$. 
	Clearly $q \le p$ and $q \in M$.

	Define $s$ as follows. 
	The domain of $s$ equals the domain of $q$, 
	$s(i_k) := z_k$ for all $k < n$, and $s(j) := q(j)$ 
	for all other $j$ in $\dom(q)$. 
	Now fix $t \le s$ in $D$. 
	By extending further if necessary, we may assume without loss of generality 
	that the elements $t(i_0),\ldots,t(i_k)$ all have the same height.

	Since $R$ is free, the tree 
	$S := R_{z_0 \res \gamma} \otimes \cdots \otimes R_{z_{n-1} \res \gamma}$ is Suslin. 
	Define $X$ to be the set of all $n$-tuples $( y_k : k < n )$ in $S$ 
	for which there exists some $r \in D$ such that 
	$r \le q$ and $r(i_k) = y_k$ for all $k < n$. 
	Note that $X \in M$ by elementarity. 
	Also observe that $t$ is a witness to the fact that 
	$( t(i_k) : k < n )$ is in $X$, and this $n$-tuple is greater than or equal to 
	$( z_k : k < n )$ in $S$.

	By Lemma 2.1 applied to the Suslin tree $S$, 
	there exists some 
	$( y_k : k < n )$ in $X \cap M$ which is below 
	$( z_k : k < n )$ in $S$. 
	Fix a witness $r \in D \cap M$ with $r \le q$ such that 
	$r(i_k) = y_k$ for all $k < n$. 
	Then $r$ is as required.
\end{proof}

\begin{thm}
	The forcing poset $\pro$ is proper and countably distributive.
\end{thm}

\begin{proof}
	Let $\theta$ be a large enough regular cardinal and consider a countable 
	elementary substructure $M \prec H(\theta)$ with $\pro \in M$. 
	Let $\delta := M \cap \omega_1$. 
	We will show that for all $p \in M \cap \pro$, there exists some $q \le p$ 
	such that for every dense open subset $D$ of $\pro$ in $M$, 
	there exists some $s \in M \cap D$ with $q \le s$. 
	By standard arguments, 
	it follows that $\pro$ is proper and countably distributive.

	Let $\langle D_n : n < \omega \rangle$ 
	enumerate all dense open subsets of 
	$\pro$ in $M$ and let 
	$\langle i_k : k < \omega \rangle$ enumerate $M \cap \kappa$. 
	Consider $p \in M \cap \pro$. 
	By induction we define:
	\begin{enumerate}
		\item a descending sequence of conditions $\langle p_n : n < \omega \rangle$;
		\item a sequence $\langle z_n : n < \omega \rangle$ of elements of 
		$R_\delta$.
	\end{enumerate}
	The following inductive hypotheses will be satisfied for all $n$:
	\begin{enumerate}
		\item[(a)] $p_{n+1} \in D_n$;
		\item[(b)] $i_n \in \dom(p_{n+1})$;
		\item[(c)] for all $m \ge n+1$, 
		$p_m(i_n) <_R z_n$.
	\end{enumerate}

	Let $p_0 := p$. 
	Now consider $n < \omega$ and assume that $p_n$ is defined together with 
	$z_k$ for all $k < n$. 
	Let $p_n'$ be equal to $p_n$ if $i_n \in \dom(p_n)$, and otherwise 
	let $p_n' := p_n \cup \{ (i_n,x_{\text{root}}) \}$ where $x_{\text{root}}$ 
	is the root of $R$. 
	Note that for all $k < n$, $p_n'(i_k) = p_n(i_k) <_R z_k$. 
	Since $R$ is normal, we can 
	fix some $z_n \in R_\delta$ such that $p_n'(i_n) <_R z_n$ and $z_n$ is 
	different from $z_k$ for all $k < n$. 
	Applying Lemma 3.7, fix $p_{n+1} \le p_n'$ in $M \cap D_n$ such that for all 
	$k \le n$, $p_{n+1}(i_k) <_R z_k$.

	This completes the construction. 
	Define $q$ as follows. 
	Let the domain of $q$ be equal to $M \cap \kappa$, and for 
	each $n < \omega$ define $q(i_n) := z_n$. 
	By Lemma 3.6, $q$ is a condition and $q \le p_n$ for all $n < \omega$. 
	Clearly, $q$ is as required.
\end{proof}

\begin{corollary}
	Assuming \textsf{CH}, $\pro$ is proper and $\omega_2$-c.c., and hence 
	preserves all cardinals.
\end{corollary}

If \textsf{CH} does not hold, then $\pro$ will collapse $2^\omega$ to become $\omega_1$, 
provided that $\ka \ge \omega_1$.

\begin{proposition}
	Assuming $\ka \ge \omega_1$, 
	the forcing poset $\pro$ forces $\textsf{CH}$.
\end{proposition}

\begin{proof}
	For each ordinal $\omega \le \alpha < \omega_1$, fix a bijection 
	$g_\alpha : \omega \to R_\alpha$. 
	For each $x \in R$, let $h_x$ be a bijection from $\omega$ onto 
	the set of elements of $R_{\h_R(x) + \omega}$ which are above $x$.

	For each ordinal $\omega \le \alpha < \omega_1$, define a $(\pro)$-name 
	$\dot f_\alpha$ for a function from $\omega$ to $\omega$ as follows. 
	Consider $n < \omega$, and let $x := g_\alpha(n)$. 
	Let $i$ be a name for the 
	least ordinal in $\ka$ such that $x \in \dot b_i$, which exists by Lemma 3.4.  
	Define $\dot f_\alpha(n) := m$, where $\dot b_i(\alpha+\omega) = h_x(m)$.

	We claim that $\pro$ forces that the sequence 
	$\langle \dot f_\alpha : \omega \le \alpha < \omega_1 \rangle$ includes 
	every function from $\omega$ to $\omega$. 
	So let $f : \omega \to \omega$ and $p \in \pro$ be given. 
	Fix $\omega \le \alpha < \omega_1$ greater than the height of every 
	member of the range of $p$. 
	Let $\langle i_n : n < \omega \rangle$ enumerate the first $\omega$-many 
	elements of $\omega_1 \setminus \dom(p)$ (which is possible since $\ka \ge \omega_1$). 
	Define $q$ so that $q \res \dom(p) := p$ and for all $n$, 
	$q(i_n) := x_{\text{root}}$ where $x_{\text{root}}$ is the root of $R$. 
	Now define $r \le q$ with the same domain as $q$ so that 
	$r(i_n) := g_{\alpha}(n)$ for all $n$, and $r(j)$ is some element of $R_\alpha$ 
	above $q(j)$ for all other $j$ in the domain of $q$.
	
	For each $n$, let $j_n$ be the least member of $\dom(r)$ such that 
	$r(j_n) = g_\alpha(n)$ (which exists because of the choice of $r(i_n)$). 
	It is easy to see that $r$ forces that $j_n$ is the minimal $j \in \ka$ 
	such that $g_\alpha(n) \in \dot b_j$. 
	Now define $s$ with the same domain as $r$ by letting 
	$s(j_n) := h_{g_\alpha(n)}(f(n))$ for all $n$, and $s(j) := r(j)$ 
	for all other $j \in \dom(r)$. 
	Then easily $s$ forces that $\dot f_\alpha = f$.
\end{proof}

\begin{proposition}
	Assuming that $\ka \ge \omega_1$, 
	the forcing poset $\pro$ forces that for every uncountable 
	downwards closed subset 
	$W \subseteq R$, there exists some $i < \ka$ such that 
	$\dot b_i \subseteq W$.
\end{proposition}

\begin{proof}
	Suppose for a contradiction that $p$ is a condition which 
	forces that $\dot W$ is an uncountable  
	downwards closed subset of $R$ which does not contain the cofinal branch 
	$\dot b_i$ for all $i < \ka$. 
	Since $\dot W$ is forced to be downwards closed, it easily follows that 
	$p$ forces that for all $i < \ka$, 
	there exists an ordinal $\dot \delta_i < \omega_1$ 
	such that every member of $\dot W$ with height at least 
	$\dot \delta_i$ is not in $\dot b_i$.

	Fix a large enough regular cardinal $\theta$, and let 
	$M$ be a countable elementary substructure of $H(\theta)$ which contains 
	as members the objects 
	$\pro$, $p$, $\dot W$, and $\langle \dot \delta_i : i < \ka \rangle$. 
	Let $\delta := M \cap \omega_1$. 
	Enumerate $R_\delta$ as $\langle y_n : n < \omega \rangle$ and enumerate 
	$M \cap \ka$ as $\langle j_n : n < \omega \rangle$.

	We will define by induction sequences $\langle p_n : n < \omega \rangle$, 
	$\langle m_n : n < \omega \rangle$, 
	$\langle i_n : n < \omega \rangle$, 
	$\langle z_n : n < \omega \rangle$, and 
	$\langle \delta_n : n < \omega \rangle$ as follows. 
	Let $p_0 := p$ and $m_0 := 0$.

	Let $n$ be given and assume that $p_n$ and $m_n$ are defined, as well as 
	$i_k$, $z_k$, and $\delta_k$ for all $k < m_n$. 
	We assume as an inductive hypothesis that 
	$p_n(i_k) <_R z_k$ for all $k < m_n$.

	The next step of the construction will consist of two stages. 
	In the first stage, we consider $y_n$. 
	
	\bigskip

	\noindent Case 1: $y_n \notin \{ z_k : k < m_n \}$. 
	Choose some ordinal $i_{m_n} \in M \cap \ka$ which is not in $\dom(p_n)$ 
	(which is possible since $\ka \ge \omega_1$) and 
	define $p_n' := p_n \cup \{ (i_{m_n},x_{\text{root}}) \}$, 
	where $x_{\text{root}}$ is the root of $R$. 
	Define $z_{m_n} := y_n$. 
	Define $m_n' := m_n+1$.
	
	\bigskip
	
	\noindent Case 2: $y_n = z_k$ for some $k < m_n$. 
	Define $p_n' := p_n$ and $m_n' := m_n$.
	
	\bigskip

	Note that in either case, for all $k < m_n'$, 
	$i_k \in \dom(p_n')$ and $p_n'(i_k) <_R z_k$. 

	In the second stage, we consider $j_n$. 
	
	\bigskip
	
	\noindent Case a: $j_n \in \{ i_k : k < m_n' \}$. 
	Define $m_{n+1} := m_n'$ and $p_n^* := p_n'$.

	\bigskip
	
	\noindent Case b: $j_n \in \dom(p_n') \setminus \{ i_k : k < m_n' \}$. 
	Define $i_{m_n'} := j_n$. 
	Let $z_{m_n'}$ be some element of $R_{\delta}$ such that 
	$p_n'(j_n) <_R z_{m_n'}$ and which is not equal to $z_k$ for all $k < m_n'$ 
	(which is possible since $R$ is normal). 
	Define $m_{n+1} := m_n'+1$ and $p_n^* := p_n'$.
	
	\bigskip
	 
	\noindent Case c: $j_n \notin \dom(p_n')$. 
	Define $i_{m_n'} := j_n$. 
	Let $z_{m_n'}$ be some element of $R_\delta$ which is not equal to 
	$z_k$ for all $k < m_n'$. 
	Define $p_n^* := p_n' \cup \{ (j_n,x_{\text{root}}) \}$, 
	where $x_{\text{root}}$ is the root of $R$. 
	Define $m_{n+1} := m_n' + 1$.

	\bigskip

	Note that for all $k < m_{n+1}$, $i_k \in \dom(p_n^*)$ 
	and $p_n^*(i_k) <_R z_k$. 

	Now apply Lemma 3.7 to fix $p_{n+1} \le p_n^*$ in $M$ which decides 
	$\dot \delta_{i_k}$ as some ordinal $\delta_k \in M$, for each 
	$m_n \le k < m_{n+1}$, 
	and satisfying that for all $k < m_{n+1}$, 
	$p_{n+1}(i_k) <_R z_k$.

	This completes the definition of the sequences. 
	It is easy to check by cases that every $j_n$ is equal to $i_k$ for some $k$, 
	that is, $\{ i_k : k < \omega \} = \{ j_n : n < \omega \} = M \cap \ka$, 
	and whenever $i_k \in \dom(p_m)$, then $p_m(i_k) <_R z_k$. 
	Moreover, every $y_n$ equals $z_k$ for some $k$, that is, 
	$\{ z_k : k < \omega \} = \{ y_n : n < \omega \} = R_\delta$.  
	Define $q$ with domain equal to $M \cap \ka$ such that for all $n$, 
	$q(i_n) := z_n$. 
	By Lemma 3.6, $q$ is a condition and $q \le p_n$ for all $n$.
	
	Now for all $n$, $q$ forces that $z_n \in \dot b_{i_n}$ and 
	$\dot \delta_{i_n} = \delta_n < \delta$. 
	Extend $q$ to $r$ which decides for some $n$ that $z_n \in \dot W$, which 
	is possible since $\dot W$ is forced to be uncountable and downwards closed. 
	Now we have a contradiction since 
	$r$ forces that every member of $\dot W$ with height at least 
	$\dot \delta_{i_n} = \delta_n$ is not in $\dot b_{i_n}$, but 
	on the other hand $r$ forces that $z_n \in \dot b_{i_n} \cap \dot W \cap R_\delta$.
\end{proof}

\begin{corollary}
	Assuming that $\ka \ge \omega_1$, 
	the forcing poset $\pro$ forces that $R$ contains no Aronszajn subtree.
\end{corollary}

\begin{proof}
	Immediate from Lemma 2.2 and Proposition 3.11.
\end{proof}

\begin{corollary}
	Assuming that $\ka \ge \omega_1$, 
	the forcing poset $\pro$ forces that every cofinal branch of $R$ is 
	equal to $\dot b_i$ for some $i < \ka$.
\end{corollary}

\begin{proof}
	Immediate from Proposition 3.11 and the fact that 
	any cofinal branch of $R$ is a downwards closed subset of $R$.
\end{proof}

Recall that a tree $T$ is \emph{rigid} if it has no automorphism other than 
the identity function, and is \emph{totally rigid} if for all distinct 
$x$ and $y$ of $T$, $T_x$ and $T_y$ are not isomorphic. 
We will show that after forcing with $\pro$ 
the tree $R$ is totally rigid, and in fact, has an even 
stronger rigidity property.

A map $f : S \to T$ between trees is \emph{strictly increasing} 
if $x <_S y$ implies $f(x) <_T f(y)$. 
In general, this property is strictly weaker than the property when ``implies'' 
is replaced with ``iff''. 
Suppose that $T$ has a cofinal branch $b$. 
Then there always exists a (trivial) strictly increasing 
map $f : S \to T$, namely, the map $f(x) := b(\h_T(x))$ for all $x \in S$. 
Let us say that a strictly increasing map $f : S \to T$ 
\emph{maps $S$ into a branch} if $f(x)$ and $f(y)$ are comparable for all $x, y \in S$. 
In this case, $f[S]$ is an uncountable chain of $T$.

\begin{definition}
	An $\omega_1$-tree $T$ is \emph{essentially rigid} if 
	for all incomparable elements 
	$x$ and $y$ of $T$ and any dense subset $U$ of $T_x$, 
	if $f : U \to T_y$ is strictly increasing, 
	then there are densely many $z \in U$ such that 
	$f$ maps $T_z \cap U$ into a branch of $T_y$.
\end{definition}

\begin{thm}
	The forcing poset $\pro$ forces that $R$ is essentially rigid.
\end{thm}

\begin{proof}
	Let $x$ and $y$ be incomparable elements of $R$. 
	Fix a generic filter $G$ on $\pro$. 
	In $V[G]$, let $U$ be a dense subset of $R_x$ and 
	let $f : U \to R_y$ be strictly increasing. 
	Suppose for a contradiction that there exists some $x^* \in U$ such that 
	for all $z >_R x^*$ in $U$, $f$ does not map $R_{z} \cap U$ 
	into a branch, which means that 
	there exist distinct elements $a$ and $b$ 
	of $R_z \cap U$ such that $f(a)$ and $f(b)$ are incomparable.

	Consider the forcing poset 
	$\prod_{\omega,\kappa+1} R$, which is isomorphic to the two-step product 
	$(\pro) \times R$. 
	Fix a $V[G]$-generic filter $H$ on $R$ with $x^* \in H$. 
	Then $H$ is a cofinal branch of $R_x$, and in $V[G][H]$, $H = b_\ka$. 
	Since $U \in V[G]$ is dense in $R_x$ and $H$ is $V[G]$-generic, 
	$U \cap H$ is uncountable. 
	Let $b$ be the downward closure of $f[H \cap U]$, which is a cofinal branch of $R_y$. 
	By Corollary 3.13 applied to the forcing 
	$\prod_{\omega,\kappa+1} R$, $b$ is equal to $b_i$ for some $i \le \ka$. 
	But $b$ contains $y$, and hence cannot contain $x$ since $x$ and 
	$y$ are incomparable. 
	So $b \ne b_\ka$. 
	Therefore, $b$ is equal to $b_i$ for some $i < \ka$.

	In the model $V[G]$, 
	define $D$ as the set of all $c \in R_{x} \cap U$ such that 
	$f(c) \notin b_i$. 
	We claim that $D$ is dense below $x^*$. 
	Consider $z >_R x^*$. 
	Since $U$ is dense in $R_x$, we may assume without loss of generality 
	that $z \in U$. 
	By the choice of $x^*$, 
	there exist distinct elements $a$ and $b$ 
	of $R_z \cap U$ such that $f(a)$ and $f(b)$ are incomparable. 
	Then $f(a)$ and $f(b)$ cannot both be in the branch $b_i$. 
	Let $c \in \{ a, b \}$ be such that $f(c) \notin b_i$. 
	This completes the argument that $D$ is dense below $x^*$. 
	Since $x^* \in H$ and $H$ is $V[G]$-generic, there exists some $c \in H \cap D$. 
	Then $f(c) \notin b_i$ and $f(c) \in f[H \cap U] \subseteq b_i$, 
	which is a contradiction.
\end{proof}

In \cite{krueger} we proved that if $R$ is a free Suslin tree, then for any 
$n < \omega$, there exists a c.c.c.\ forcing which forces that $R$ 
is $n$-free but for all $m > n$, every $m$-derived tree is special. 
Combined with the results of this article, we see that free Suslin trees can 
acquire a wide variety of properties by forcing.

As mentioned in the introduction, there are limitations to turning a free 
Suslin tree into a Kurepa tree via c.c.c.\ forcing. So the best 
we can ask for is a consistency result.

\begin{question}
Is it consistent that there exists a free Suslin tree and 
for any normal free Suslin tree $R$, there exists a c.c.c.\ forcing poset 
which forces that $R$ is a Kurepa tree?
\end{question}

\section*{Funding}

This material is based upon work supported by the Simons Foundation under Grant 631279.

\nocite{jech}

\nocite{devlin}

\end{document}